\documentclass[12pt]{amsart}
\usepackage[T2A,T1]{fontenc}
\usepackage[utf8]{inputenc}		
\usepackage{ucs}
\usepackage[russian,english]{babel}	
\usepackage{amssymb}
\usepackage{ dsfont }
\usepackage{multirow}
\usepackage{amsthm}
\usepackage{tikz}
\usepackage{amsmath}
\usepackage{graphicx}
\usepackage{wrapfig}
\usepackage{caption}
\usepackage[left=2.5cm,right=2.5cm,top=2.4cm,bottom=2.4cm,bindingoffset=0cm]{geometry}
\usepackage{enumerate}
\usepackage{makecell}	
\newtheorem{state}{Statement}[section]

\newtheorem*{theorem}{Main Theorem}

\newtheorem{lemm}[state]{Lemma}
\newtheorem{lemma}[state]{Lemma}

\makeatletter
\newcommand*{\rom}[1]{\expandafter\@slowromancap\romannumeral #1@}
\makeatother

\makeatletter
\def\@seccntformat#1{\csname the#1\endcsname. }
\def\@biblabel#1{#1.}

\makeatother

\title[On characterization of groups by isomorphism type of Gruenberg--Kegel graph]{On characterization of groups by isomorphism type of Gruenberg--Kegel graph}

\author{Mingzhu Chen}
\address{Mingzhu Chen \newline School of Mathematics and Statistics, Hainan University,\newline
Haikou, 570225, Hainan, P. R. China}
\email{mzchen@hainanu.edu.cn}

\author{Natalia~V.~Maslova}
\address{Natalia Vladimirovna Maslova \newline Krasovskii Institute of Mathematics and Mechanics UB RAS,\newline
16, S. Kovalevskaja str.,
Yekaterinburg, 620108, Russia\newline
Ural Federal University,\newline
19, Mira str., Yekaterinburg, 620002, Russia \newline
ORCID: 0000-0001-6574-5335}
\email{butterson@mail.ru}

\author{Marianna~R.~Zinov'eva}
\address{Marianna Rifhatovna Zinov'eva \newline Krasovskii Institute of Mathematics and Mechanics UB RAS,\newline
16, S. Kovalevskaja str.,
Yekaterinburg, 620108, Russia\newline
Ural Federal University,\newline
19, Mira str., Yekaterinburg, 620002, Russia \newline
ORCID: 0009-0006-0560-2192}
\email{zinovieva-mr@yandex.ru}

\begin{document}

\begin{abstract}
The Gruenberg--Kegel graph (or the prime graph) $\Gamma(G)$ of a finite group $G$ is the graph whose vertex set is the set of prime divisors of $|G|$ and in which two distinct vertices $r$ and $s$ are adjacent if and only if there exists an element of order $rs$ in $G$. We say that a group $G$ is recognizable by Gruenberg--Kegel graph if for every group $H$ the equality $\Gamma(H)=\Gamma(G)$ implies that $G\cong H$. A group $G$ is recognizable by isomorphism type of Gruenberg--Kegel graph if for every group $H$ the isomorphism between $\Gamma(H)$ and $\Gamma(G)$ as abstract graphs (i.\,e. unlabeled graphs) implies that $G\cong H$. In 2022, P.~J.~Cameron and the second author proved that if a finite group is recognizable by Gruenberg--Kegel graph, then the group is almost simple. It is clear that if a group is recognizable by isomorphism type of Gruenberg--Kegel graph, then the group is recognizable by Gruenberg--Kegel graph.  There are still not so many examples of groups which can be recognizable by isomorphism type of Gruenberg--Kegel graph. In 2006, A.~V.~Zavarnitsine proved that finite simple sporadic group $J_4$ is the unique finite group with exactly $6$ connected components of its Gruenberg--Kegel graph. In 2022, P.~J.~Cameron and the second author proved that the groups $E_8(2)$ and ${^2}G_2(27)$ are recognizable by isomorphism type of Gruenberg--Kegel graph, and recently M.~Lee and T.~Popiel proved that a sporadic simple group $S$ is recognizable by isomorphism type of Gruenberg--Kegel graph if and only if $S \in \{\mathbb{B}, Fi_{23}, Fi'_{24}, J_4, Ly, \mathbb{M}, O’N, Th\}$. In this paper, we prove that finite simple exceptional groups of Lie type ${^2}E_6(2)$ and $E_8(q)$ for $q \in \{3, 4, 5, 7, 8, 9, 17\}$ are recognizable by isomorphism type of Gruenberg--Kegel graph.

\medskip


\medskip

MSC classes: 20D60, 20D06, 05C99.

\medskip

\hfill Dedicated to Cheryl~E. Praeger on the occasion of her anniversary

\end{abstract}

\maketitle

Throughout the paper we consider only finite groups and simple graphs, and henceforth the term group means finite group and the term graph means simple graph, that is undirected graph without loops and multiple edges.

Let $G$ be a group. The {\it spectrum} $\omega(G)$ is the set of all element orders of $G$. The {\it prime spectrum} $\pi(G)$ is the set of all primes belonging to $\omega(G)$. A graph $\Gamma(G)$ whose vertex set is $\pi(G)$ and in which two distinct vertices~$r$ and~$s$ are adjacent if and only if $rs \in \omega(G)$ is called the {\it Gruenberg--Kegel graph} or the {\it prime graph} of~$G$.

Let $G$ and $H$ be groups. We write $\Gamma(G)=\Gamma(H)$ if $\pi(G)=\pi(H)$ and two distinct primes $p$ and $q$ are adjacent in $\Gamma(G)$ if and only if they are adjacent in $\Gamma(H)$. We write $\Gamma(G) \cong \Gamma(H)$ if $\Gamma(G)$ and $\Gamma(H)$ are isomorphic as abstract graphs (i.\,e. unlabeled graphs).

\medskip

\textbf{Example} {\rm (see \cite[P.~188]{CaMas})}. For groups $G$ and $H$,

\begin{itemize}

\item if $G \cong H$, then $\omega(G)=\omega(H)${\rm;}

\item if $\omega(G)=\omega(H)$, then $\Gamma(G)=\Gamma(H)${\rm;}

\item if $\Gamma(G)=\Gamma(H)$, then $\pi(G)=\pi(H)$ and $\Gamma(G) \cong \Gamma(H)$.

\end{itemize}

The converse does not hold in each case, as the following series of examples demonstrates{\rm:}

\begin{itemize}

\item $S_5\not \cong S_6$ but $\omega(S_5)=\omega(S_6)${\rm;}

\item $\omega(A_5)\not =\omega(A_6)$ but $\Gamma(A_5)=\Gamma(A_6)${\rm;}

\item $\Gamma(A_{10})\not =\Gamma(Aut(J_2))$ but  $\Gamma(A_{10})$ and $\Gamma(Aut(J_2))$ are isomorphic as abstract graphs  and $\pi(A_{10}) =\pi(Aut(J_2))$, see Figure~1.

    \begin{center}
    \begin{tikzpicture}
    \tikzstyle{every node}=[draw,circle,fill=white,minimum size=4pt,
                            inner sep=0pt]
    \draw (0,0) node (2) [label=left:$2$]{}
        -- ++ (-30:1cm) node (3) [label=below:$3$]{}
        -- ++ (0:1cm) node (7) [label=below:$7$]{}
          (-90:1.0cm) node (5) [label=left:$5$]{};
    \draw (5) -- (3);
    \draw (5) -- (2);
    \end{tikzpicture}{$\Gamma(A_{10})$;} $\mbox{ }$$\mbox{ }$$\mbox{ }$$\mbox{ }$$\mbox{ }$$\mbox{ }$$\mbox{ }$$\mbox{ }$
    \begin{tikzpicture}
    \tikzstyle{every node}=[draw,circle,fill=white,minimum size=4pt,
                            inner sep=0pt]
    \draw (0,0) node (2) [label=left:$3$]{}
        -- ++ (-30:1cm) node (3) [label=below:$2$]{}
        -- ++ (0:1cm) node (7) [label=below:$7$]{}
          (-90:1.0cm) node (5) [label=left:$5$]{};
    \draw (5) -- (3);
    \draw (5) -- (2);
    \end{tikzpicture}{$\Gamma(Aut(J_2))$.}\\
    Figure~1
\end{center}

\smallskip

\end{itemize}

\noindent We say that the group $G$ is
\begin{itemize}
\item{\it recognizable} by its spectrum (Gruenberg--Kegel graph, respectively) if for each group~$H$, $\omega(G)=\omega(H)$ ($\Gamma(G)=\Gamma(H)$, respectively) if and only if $G \cong H${\rm;}
\item  {\it $k$-recognizable} by spectrum (Gruenberg--Kegel graph, respectively), where $k$ is a positive integer number, if there are exactly $k$ pairwise non-isomorphic groups with the same spectrum (Gruenberg--Kegel graph, respectively) as $G${\rm;}
\item  {\it almost recognizable} by spectrum (Gruenberg--Kegel graph, respectively) if it is $k$-recognizable by spectrum (Gruenberg--Kegel graph, respectively) for some positive integer number $k${\rm;}
\item {\it unrecognizable} by spectrum (Gruenberg--Kegel graph, respectively), if there are infinitely many pairwise non-isomorphic groups with the same spectrum (Gruenberg--Kegel graph, respectively) as $G$.
\end{itemize}

\medskip

The problem of characterization of a finite group by spectrum is well-known and was widely researched. There are strong and nice results in this research area; in particular, 'almost all' finite simple groups are almost recognizable by spectrum, and each finite simple group is uniquely determined up to isomorphism by its spectrum and order. A survey of this research area can be found in~\cite{GreMazShiVasYang} with updates in~\cite[Section~2]{Maslova_surv}. It is clear that if a finite group is (almost) recognizable by Gruenberg--Kegel graph, then this group is (almost) recognizable by spectrum (the converse does not hold in the general case) and the spectrum gives much more information about a group than its Grunberg--Kegel graph. On the other hand, each group which is almost recognizable by spectrum is almost simple~\cite[Theorem~1.3]{CaMas} while almost simple groups form an important class of finite groups~\cite[Section~3]{Maslova_surv}; a survey of recent progress in characterization of a finite group by Gruenberg--Kegel graph can be found in~\cite{CaMas} and~\cite[Section~3]{Maslova_surv}.

\smallskip

It turns out that sometimes when we discuss the question of characterization of a group by Gruenberg--Kegel graph we can 'forget' about labels of the vertices of Grueberg--Kegel graph. We say that the group $G$ is
\begin{itemize}
\item {\it recognizable by isomorphism type of Gruenberg--Kegel graph} if for each group~$H$, $\Gamma(G)\cong \Gamma(H)$ if and only if $G \cong H${\rm;}
\item  {\it $k$-recognizable by isomorphism type of Gruenberg--Kegel graph}, where $k$ is a positive integer number, if there are exactly $k$ pairwise non-isomorphic groups $H$ with $\Gamma(H) \cong \Gamma(G)${\rm;}
\item  {\it almost recognizable by isomorphism type of Gruenberg--Kegel graph} if it is $k$-recognizable by isomorphism type of Gruenberg--Kegel graph for some positive integer number $k${\rm;}
\item {\it unrecognizable by isomorphism type of Gruenberg--Kegel graph} if there are infinitely many pairwise non-isomorphic groups $H$ with $\Gamma(H) \cong \Gamma(G)$.
\end{itemize}

\medskip

By~\cite[Theorem~1.3]{CaMas}, if a group $G$ is recognizable by isomorphism type of Gruenberg--Kegel graph, then $G$ is almost simple.  There are not so many results about characterization of a group by isomorphism type of Gruenberg--Kegel graph. In~2006, A.~V.~Zavarnitsine has proved that sporadic simple group $J_4$ is recognizable by isomorphism type of Gruenberg--Kegel graph~\cite[Theorem~B]{Zavarnitsine_2006}. Recently M.~Lee and T.~Popiel~\cite{LeePopiel_2} proved that a sporadic simple group $S$ is recognizable by isomorphism type of Gruenberg--Kegel graph if and only if $S \in \{\mathbb{B}, Fi_{23}, Fi'_{24}, J_4, Ly, \mathbb{M}, O’N, Th\}$; all other sporadic simple groups are unrecognizable by isomorphism type of Gruenberg--Kegel graph. In 2013, A.~V.~Zavarnitsine~\cite{Zavarnitsine_2013} proved that if $S$ is a group such that $\Gamma(S)$ has exactly $5$ connected components, then $S\cong E_8(q)$, where $q \equiv 0, \pm 1 \pmod{5}$. In 2022, P.~J.~Cameron and the second author~\cite[Theorem~1.5]{CaMas} proved that the groups $E_8(2)$ and ${^2}G_2(27)$ are recognizable by isomorphism type of Gruenberg--Kegel graph.

\medskip

In this paper we investigate the question of characterization by isomorphism type of Gruenberg--Kegel graph of groups ${^2}E_6(2)$ and $E_8(q)$ for $q \in \{3, 4, 5, 7, 8, 9, 17\}$. In 2021, A.~S.~Kondrat'ev~\cite{Kondrat'ev2E6(2)} proved that ${^2}E_6(2)$ is recognizable by Gruenebrg--Kegel graph, and recently the second author, V.~V.~Panshin and A.~M.~Staroletov~\cite[Theorem~6.1]{MasPansStar} continuing the research~\cite{Zavarnitsine_2013} proved that if $\Gamma(S)=\Gamma(E_8(q))$, where $q \equiv \pm 2\pmod{5}$, then $S \cong E_8(u)$ for some prime power $u$ with $u \equiv \pm 2\pmod{5}$. Thus, each group $S=E_8(q)$ is almost recognizable by Gruenberg--Kegel graph by~\cite[Proposition~3.1]{CaMas}, however, even the question whenever $S$ is recognizable by Gruenberg--Kegel graph is still open in the general case, see~\cite[Problem~6.3]{MasPansStar}.

\medskip

We prove the following theorem.

\medskip

\begin{theorem}\label{Main} $(i)$ Let $G$ be a group such that $\Gamma(G)$ is isomorphic to the following graph

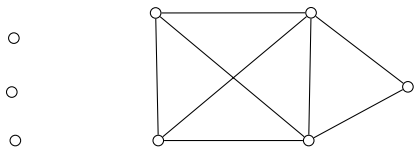
\begin{center}    \begin{tikzpicture}
        \tikzstyle{every node}=[draw,circle,fill=white,minimum size=4pt, inner sep=0pt]
        \draw (0,0) node (3)  {}
        (-3.2cm:-1.7cm) node (2)  {}
        (0.0cm:-2.0cm) node (5)  {}
        (-1.4cm:-2.65cm) node (7)  {}
        (1.0cm:1.5cm) node (11)  {}

        (6.0cm:4.0cm) node (17)  {}
        (0.0cm:-3.9cm) node (19)  {}
        (-7.0cm:4.15cm) node (13)  {}
 (3)--(11)
 (2)--(11)
 (2)--(3)
 (2)--(5)
 (2)--(7)
 (3)--(5)
 (3)--(7)
 (5)--(7)
 ;
    \end{tikzpicture}\\
    \end{center}
Then $G \cong {^2}E_6(2)$.

$(ii)$ Let $G=E_8(q)$ for $q \in \{3, 4, 5, 7, 8, 9, 17\}$. If $H$ is a group such that $\Gamma(H) \cong \Gamma(G)$, then $H \cong G$.

\smallskip

On the other words, the groups ${^2}E_6(2)$ and $G=E_8(q)$ for $q \in \{3, 4, 5, 7, 8, 9, 17\}$  are recognizable by isomorphism type of Gruenberg--Kegel graph.

\end{theorem}

This paper contains a revised and improved proof of Main Theorem for groups $E_8(q)$; the original proof based on more calculations was obtained by the second author with some participation of the third author; and announcement can be found, for example, in~\cite{Maslova_UWGTC2020}.

The following problems naturally arise.

\medskip

{\bf Problem~1.} {\it Are there other finite simple groups which are recognizable by isomorphism types of Gruenberg--Kegel graph{\rm?}}

\medskip

{\bf Problem~2.} {\it Is there an almost simple but not simple group which is recognizable by isomorphism type of Gruenberg--Kegel graph{\rm?}}

\section{Preliminaries}

If $n$ is an integer and $r$ is an odd prime with $(r, n) = 1$, then $e(r, n)$ denotes the multiplicative order of $n$ modulo $r$. Given an odd integer $n$, we put $e(2, n) = 1$ if $n\equiv1\pmod{4}$, and $e(2,n)=2$ otherwise.

The following lemma is proved in~\cite{Bang}, and also in~\cite{Zs92}.
\begin{lemm}[{\rm Bang–Zsigmondy}]\label{zsigm}
Let $q$ be an integer greater than $1$. For every positive integer $m$ there exists a prime $r$ with $e(r,q)=m$ besides the cases $q=2$ and $m=1$, $q=3$ and $m=1$, and $q=2$ and $m=6$.
\end{lemm}	

Fix an integer $a$ with $|a|>1$. A prime $r$ is said to be a primitive prime divisor of $a^i-1$ if $e(r,a)=i$. We write $r_i(a)$ to denote some primitive prime divisor of $a^i-1$ if such a prime exists, and $R_i(a)$ to denote the set of all such divisors.

\begin{lemm}[{\rm \cite{Herzog}}]\label{PrimesSmall} Let $q$ be a prime power. Then $|\pi(q^2-1)|\le 2$ if and only if $q \in \{2, 3, 4, 5, 7, 8, 9, 17\}$.
\end{lemm}

Let $G$ and $H$ be groups, $n$ be a positive integer, $p$ be a prime and $\pi$ be a set of primes. We denote by $S(G)$ the {\it solvable radical} of $G$ (the largest solvable normal subgroup of $G$), by $F(G)$ the {\it Fitting subgroup} of $G$ (the largest nilpotent normal subgroup of $G$), by $\Phi(G)$ the {\it Frattini subgroup} of $G$ (the intersection of all maximal subgroups of $G$), and by $Soc(G)$ the {\it socle} of $G$ (the subgroup of $G$ generated by the set of all non-trivial minimal normal subgroups of $G$). By $G.H$ we denote any extension of $G$ by $H$, by $G:H$ (or $G \rtimes H$) we denote a split extension (or semidirect product) of $G$ by $H$. By $\pi(n)$ we denote the set of all prime divisors of $n$; in this notation, $\pi(G)=\pi(|G|)$. We tell that $G$ is a $\pi$-group if $\pi(G) \subseteq \pi$ and $G$ is a $p$-group if $\pi(G)=\{p\}$. By $O_\pi(G)$ and $O_p(G)$ we denote the largest normal $\pi$-subgroup and the largest normal $p$-subgroup of $G$, respectively. By $IBr_p(G)$ we denote the set of irreducible $p$-Brauer characters of $G$. Denote the number of connected components of $\Gamma(G)$ by $s(G)$, and the set of connected components of $\Gamma(G)$ by $\{\pi_i(G) \mid 1 \leq i \leq s(G) \}$; for a group $G$ of even order, we assume that $2 \in \pi_1(G)$. Denote by $t(G)$ the \emph{independence number} of $\Gamma(G)$ (the greatest cardinality of a coclique in $\Gamma(G)$), and by $t(r,G)$ the greatest cardinality of a coclique in $\Gamma(G)$ containing a prime $r$.

\medskip

The next assertion is well-known and easy-proving.

\begin{lemm}\label{NormalSeriesAdj} Let $K$ be a normal subgroup of a group $L$. Then the following conditions hold{\rm:}

$(1)$ if $r, s \in \pi(K) \setminus \pi(L/K)$ and $r$ and $s$ are non-adjacent in $\Gamma(K)$, then they are also non-adjacent in $\Gamma(L)${\rm;}

$(2)$ if $r, s \in \pi(L/K) \setminus \pi(K)$, and $r$ and $s$ are non-adjacent in $\Gamma(L/K)$, then they are also non-adjacent in $\Gamma(L)$.

\smallskip
In particular, if $A$ and $B$ are normal subgroups of a group $G$ such that $A\le B$ and $r, s \in \pi(B/A)\setminus (\pi(A) \cup \pi(G/B))$, then $r$ and $s$ are adjacent in $\Gamma(G)$ if and only if $r$ and $s$ are adjacent in $\Gamma(B/A)$.
\end{lemm}

\begin{proof}

$(1)$ Proof can be found, for example, \cite[Lemma~2]{MaslovaPrGrSub}.

$(2)$ Follows directly from the Shur--Zassenhaus Theorem~\cite[Theorem~6.2.1]{Gorenstein}.

\end{proof}

A group $G$ is called a {\it Frobenius group} if there is a subgroup $H$ of $G$ such that $H \cap H^g=1$ for each $g \in G\setminus H$. Let $$
K=\{1_G\} \cup (G \setminus ( \cup_{g \in G} H^g))$$ be the {\it Frobenius kernel} of $G$. It is well-known~\cite[35.24 and~35.25]{Asch86} that $K \trianglelefteq G$, $G=K \rtimes H$, $C_G(h)\le H$ for each $h \in H$, and $C_G(k)\le K$ for each $k \in K$. Moreover, by the Thompson theorem on finite groups with fixed-point-free automorphisms of prime order~\cite[Theorem~1]{Thompson}, $K$ is nilpotent. A $2$-Frobenius group is a group $G$ which contains a normal Frobenius subgroup $R$ with Frobenius kernel $A$ such that $G/A$ is a Frobenius group with Frobenius kernel $R/A$.

\begin{lemm}[{\rm Gruenberg--Kegel Theorem, follows from \cite[Theorem~A]{Williams} and~\cite[Theorem~1]{Thompson}}]\label{Gruenberg--Kegel theorem} If~$G$ is a group with disconnected Gruenberg--Kegel graph, then one of the following statements holds{\rm:}
\begin{itemize}
\item[$(1)$] $G$ is a Frobenius group{\rm;}
\item[$(2)$] $G$ is a $2$-Frobenius group{\rm;}
\item[$(3)$] $G$ is an extension of a nilpotent $\pi_1(G)$-group by a group~$A$, where $S \unlhd A\le Aut(S)$,~$S$ is a simple non-abelian group with $s(G)\le s(S)$, and $A/S$ is a $\pi_1(G)$-group.
\end{itemize}
\end{lemm}

\begin{lemm}[{\rm \cite[Lemma~1.1, Table~3]{ak}}]\label{Table} Let $S$ be a finite simple group with disconnected Gruenberg--Kegel graph. Then the following statements hold{\rm:}

$(1)$ for $i>1$, each $\pi_i(S)$ forms a clique in $\Gamma(S)${\rm;}

$(2)$ If $s(S)\ge 4$, then $S$ is in~Table~$1$.

\end{lemm}

\begin{table}

\centerline{{\bf Table 1.} Finite simple groups~$S$ with $s(S)>3$} \vspace{2mm}
{\footnotesize
\begin{tabular}{|p{5mm}|p{10mm}|p{18mm}|p{30mm}|p{20mm}|p{20mm}|p{20mm}|p{14mm}|p{6mm}|}
\hline
\centerline{$s(S)$} & \centerline{$S$} & \centerline{$Restrictions$} & \centerline{$\pi_1(S)$} & \centerline{$\pi_2(S)$}
& \centerline{$\pi_3(S)$} & \centerline{$\pi_4(S)$} & \centerline{$\pi_5(S)$} & \centerline{$\pi_6(S)$}\\ [3ex] \hline \centerline{$4$} & & & & & & & &\\[1ex]

& \centerline{$PSL_3(4)$} & & \centerline{$\{2\}$} & \centerline{$\{3\}$} & \centerline{$\{5\}$} & \centerline{$\{7\}$} & &\\[1ex]

& \centerline{$^2B_2(q)$} & \centerline{$q{=}2^{2m{+}1}{>}2$} & \centerline{$\{2\}$} & \centerline{$\pi(q{-}1)$} & \centerline{$\pi(q{-}\sqrt
{2q}{+}1)$} & \centerline{$\pi(q{+}\sqrt {2q}{+}1)$} & &\\[1ex]

& \centerline{$^2E_6(2)$} & & \centerline{$\{2, 3, 5, 7, 11\}$} & \centerline{$\{13\}$} & \centerline{$\{17\}$} & \centerline{$\{19\}$} & &\\[1ex]

& $E_8(q)$ &$q{\equiv} 2,3(5)$ & $\pi \big(q(q^8{-}1)(q^{12}{-}1)$ & $\pi(\displaystyle\frac{
q^{10}{+}q^5{+}1}{ q^2{+}q{+}1})$ & $\pi(q^8{-}q^4{+}1)$ & $\pi(\displaystyle\frac{
q^{10}{-}q^5{+}1}{ q^2{-}q{+}1})$ & &\\[0.3ex] & & & $(q^{14}{-}1)(q^{18}{-}1)$
& & & & &\\[0.5ex] & & & $(q^{20}{-}1)\big)$ & & & & &\\[1ex]

& \centerline{$M_{22}$} & &
\centerline{$\{2, 3\}$} & \centerline{$\{5\}$} & \centerline{$\{7\}$} & \centerline{$\{11\}$} & &\\[1ex]

& \centerline{$J_1$} & & \centerline{$\{2, 3, 5\}$} & \centerline{$\{7\}$} & \centerline{$\{11\}$} & \centerline{$\{19\}$} &
&\\[1ex]

& \centerline{$O'N$} & & $\{2, 3, 5, 7\}$ & \centerline{$\{11\}$} & \centerline{$\{19\}$} & \centerline{$\{31\}$} & &\\[1ex]

& \centerline{$Ly$} & &
$\{2, 3, 5, 7, 11\}$ & \centerline{$\{31\}$} & \centerline{$\{37\}$} & \centerline{$\{67\}$} & &\\[1ex]

& \centerline{$Fi_{24}'$} & & $\{2, 3, 5, 7,
11, 13\}$ & \centerline{$\{17\}$} & \centerline{$\{23\}$} & \centerline{$\{29\}$} & &\\[1ex]

&\centerline{$\mathbb{M}$} & & $\{2, 3, 5, 7, 11, 13, 17,$ & \centerline{$\{41\}$} &
\centerline{$\{59\}$} &\centerline{$\{71\}$}& &\\[0.5ex] & & & $ 19, 23, 29, 31, 47\}$ & & & & &\\[1ex]
\hline

\centerline{$5$} &  &  &  &  & & & &\\[3ex]

 & $E_8(q)$ &$q{\equiv} 0,1,4(5)$ & $\pi (q(q^8{-}1)(q^{10}{-}1)$ &
$\pi(\displaystyle\frac{q^{10}{+}q^5{+}1}{q^2{+}q{+}1})$ &
$\pi(\displaystyle\frac{q^{10}{-}q^5{+}1}{q^2{-}q{+}1})$ & $\pi(q^8{-}q^4{+}1)$&
$\pi(\displaystyle\frac{q^{10}{+}1}{q^2{+}1})$ &\\[0.5ex]

& & & $(q^{12}{-}1)(q^{14}{-}1)$ & & & & &\\[0.5ex]

& & & $(q^{18}{-}1))$ & & & & & \\[1ex]

\hline

\centerline{$6$} & & & & & & & & \\[1ex]

 & \centerline{$J_4$} & & \centerline{$\{2, 3, 5, 7, 11\}$} & \centerline{$\{23\}$} & \centerline{$\{29\}$} & \centerline{$\{31\}$} & \centerline{$\{37\}$} & \centerline{$\{43\}$}\\ \hline
\end{tabular}
}

\end{table}

\begin{lemm}\label{Structure} Let $L$ be a simple non-abelian group with disconnected Gruenberg--Kegel graph such that $\pi_1(L)$ is not a clique and if $s(L)=2$, then $\pi_1(L)$ is not a clique with a unique edge deleted.  Let $H$ be a group such that $\Gamma(H)$ is isomorphic to $\Gamma(L)$ and let $$\Psi: \Gamma(L) \rightarrow \Gamma(H) \mbox{ be a graph isomorphism}.$$ Then $H$ has a unique compositional factor $S$ and the following statements hold{\rm:}

$(1)$ $\Psi(\pi_1(L))=\pi_1(H)$\footnote{Here we mean that $\Psi$ is an isomorphism between induced subgraphs on $\pi_1(L)$ and $\pi_1(H)$.}{\rm;}

$(2)$ there is an injective function $$f: \{2, \ldots, s(L)\} \rightarrow \{2, \ldots, s(S)\}$$ such that $\Psi(\pi_i(L))=\pi_{f(i)}(S)$, in particular, $|\pi_i(L)|=|\pi_{f(i)}(S)|$.

\smallskip

Moreover, if $s(L)=s(S)$, then there is a one-to-one correspondence between the multi-sets $\{|\pi_i(L)| \mid 2 \le i \le s(L)\}$ and $\{|\pi_j(S)| \mid 2 \le j \le s(S)\}$.
\end{lemm}

\begin{proof} By \cite[Lemma~3, Proposition~1]{ZinMaz}, $H$ is neither a Frobenius group nor a $2$-Frobenius group. Thus, by Lemma~\ref{Gruenberg--Kegel theorem}, the factor-group $H/F(H)$ is almost simple with socle $S$.
By Lemma~\ref{Table}, $\pi_i(L)$ is a clique for each $i$ with $2 \le i \le s(L)$ and $\pi_j(S)$ is also a clique for each $j$ with $2 \le j \le s(S)$. Moreover, by Lemma~\ref{Gruenberg--Kegel theorem}, $\pi(F(H)) \cup \pi(|H/F(H):S|) \subseteq \pi_1(H)$. Therefore, by Lemmas~\ref{NormalSeriesAdj} and~\ref{Table}, $\Psi(\pi_1(L))=\pi_1(H)$ and for each $i$ with $2\le i \le s(H)=s(L)$ there is a unique $j$ with $2\le j\le s(S)$ such that $\Psi(\pi_i(L))=\pi_{j}(S)$, i.\,e., there is an injective function $$f: \{2, \ldots, s(L)\} \rightarrow \{2, \ldots, s(S)\}$$ such that $\Psi(\pi_i(L))=\pi_{f(i)}(S)$, in particular, $|\pi_i(L)|=|\pi_{f(i)}(S)|$.
Now it is clear that if $s(L)=s(S)$, then there is a one-to-one correspondence between the multisets $$\{|\pi_i(L)| \mid 2 \le i \le s(L)\} \mbox{ and } \{|\pi_j(S)| \mid 2 \le j \le s(S)\}.$$

\end{proof}

\begin{lemm}[{\rm \cite{Va05}}]\label{vas}
Let $G$ be a finite group with $t(G)\geq3$ and $t(2,G)\geq2$. Then the following statements hold.
	
$(1)$ There exists a nonabelian simple group $S$ such that $S \unlhd \overline{G} = G/S(G) \le\operatorname{Aut}(S)$.
	
$(2)$ For every coclique $\rho$ of $\Gamma(G)$ of size at least three, at most one prime in $\rho$ divides the product $|K|\cdot|\overline{G}/S|$. In particular, $t(S)\geq t(G)-1$.
	
$(3)$ One of the following two conditions holds:	
	
$\mbox{ }$$\mbox{ }$$\mbox{ }$$(3.1)$ $S\cong A_7$ or $PSL_2(q)$ for some odd $q$, and $t(S)=t(2,S)=3$.
	
$\mbox{ }$$\mbox{ }$$\mbox{ }$$(3.2)$
Every prime $p\in\pi(G)$ nonadjacent to $2$ in $\Gamma(G)$ does not divide the product $|K|\cdot|\overline{G}/S|$. In particular, $t(2,S)\geq t(2,G)$.
\end{lemm}

The next assertion is well-known and easy-proving. The proof of this assertion can be found, for example, in \cite[Lemma~1.13]{MasPansStar}.

\begin{lemma}\label{Eigenvector1} Let $G$ be a group, $g\in G$ be an element of order $r$, and $\phi$  a non-trivial irreducible representation of $G$ on a nonzero vector space $V$. If the minimum polynomial degree of $\phi(g)$ equals to $r$, then $g$ fixes in $V$ a nonzero vector.
\end{lemma}

\begin{lemm}[{\rm \cite[Theorem 1.1]{TiepZal}}]\label{TiepZalThm}
Let $G$ be one of the groups ${^2}B_2(q)$, where $q > 2$, ${^2}G_2(q)$, where $q>3$, ${^2}F_4(q)$, $G_2(q)$, ${^3}D_4(q)$. Let $g \in G$ an element of prime power order coprime to $q$. Let $\phi$ be a non-trivial irreducible representation of $G$ over a field $F$ of characteristic $l$ coprime to $q$. Then the minimum polynomial degree of $\phi(g)$
equals $|g|$, unless possibly when $G = {^2}F_4(8)$, $l= 3$, $p=109$ and $\phi(1) < 64692$.
\end{lemm}

\begin{lemm}[{\rm \cite[Lemma~4]{DoJaLu}}]\label{BrChar} Let $N$ be a nontrivial normal subgroup of a group $G$, such that $G/N \cong S$, with $S$ a simple group. If there is an element $g \in G$ of prime order that acts fixed-point-freely on $N$ then, for
every prime $r$ dividing $|N|$, there exists some $\chi \in IBr_r(S)$ such that $[\chi_T, 1_T]=0$, where $T=\langle gN \rangle$.
\end{lemm}

Let $S$ be a finite simple group of Lie type in characteristic $p$. Let $A$ be any abelian $p$-group with an $S$-action. An element $s \in S$ is said to be {\it unisingular} on $A$ if $s$ has a (nonzero) fixed point on $A$. The group $S$ is said to be {\it unisingular} if every element $s \in S$ acts unisingularly on every finite abelian $p$-group $A$ with an $S$-action.
Denote by $PSL^{\varepsilon}_n(q)$, where $\varepsilon \in\{+,-\}$, the group $PSL_n(q)$ if $\varepsilon=1$ and
$PSU_n(q)$ if $\varepsilon=-1$. Similarly, $E^{\varepsilon}_6(q)$ denotes the simple group $E_6(q)$ if $\varepsilon=1$ and ${}^2E_6(q)$ if $\varepsilon=-1$.

\begin{lemm}[{\rm\cite[Theorem 1.3]{GuTi03}}]\label{Unisingular} A finite simple group $S$ of Lie type of characteristic $p$ is unisingular if
and only if $S$ is one of the following{\rm:}

$(i)$ $PSL_n^\varepsilon(p)$ with $\varepsilon \in \{+,-\}$ and $n$ divides  $p-\varepsilon 1${\rm;}

$(ii)$ $P\Omega_{2n+1}(p)$, $PSp_{2n}(p)$ with $p$ odd{\rm;}

$(iii)$ $P\Omega_{2n}^\varepsilon (p)$ with $\varepsilon \in \{+,-\}$, $p$ odd, and $\varepsilon =(-1)^{n(p-1)/2}${\rm;}

$(iv)$ ${^2}G_2(q)$, $F_4(q)$, ${^2}F_4(q)$, $E_8(q)$ with $q$ arbitrary{\rm;}

$(v)$ $G_2(q)$ with $q$ odd{\rm;}

$(vi)$ $E_6^\varepsilon(p)$ with $\varepsilon \in \{+,-\}$ and $3$ divides $p-\varepsilon 1${\rm;}

$(vii)$ $E_7(p)$ with $p$ odd.
\end{lemm}

\begin{lemm}[{\rm \cite[Proposition~3.2]{Stewart}}]\label{PSL2qODD} Let $G$ be a group, $H\unlhd G$, $G/H \cong PSL_2(q)$, where $q>5$ is odd.
If $C_H(t)=1$ for some element $t \in G\setminus H$ such that $|t|=3$, then $H=1$.
\end{lemm}

\begin{lemm}[{\rm \cite[Lemma~1]{Mazurov}}]\label{MazurovLemma} Let $G$ be a group, $N\unlhd G$, and $G/N$ be a Frobenius group with kernel $F$ and a cyclic complement $C$. If $(|F|,|N|)=1$ and $F \not \le NC_G(N)/N$, then $s|C|\in \omega(G)$ for each $s \in \pi(N)$.

\end{lemm}

\begin{lemm}[{\rm \cite[Table~5.1.B]{KL}, \cite[Proposition~2.7]{VaVd11} and \cite[Propositions~3.2,~4.5]{VaVd05}}]\label{graphE8}
Let $G\cong  E_8(q)$, where $q$ is a power of a prime $p$. Then the following statements hold{\rm:}

$(1)$ $|G|= q^{120} \cdot \prod_{i \in \{2, 8, 12, 14, 18, 20, 24, 30\}} (q^i-1)$.

$(2)$ Suppose that $r,s\in \pi(G)$ with $r\ne s$. Then $r$ and $s$ are nonadjacent in $\Gamma(G)$ if and only if one of the following conditions holds:
	\begin{enumerate}
  \item [\text{\normalfont(1)}] $r\in\{2,p\}$, $s\ne p$ and $e(s,q)\in\{15,20,24,30\}$.
  \item [\text{\normalfont(2)}] $s,r\not\in\{2,p\}$, $k=e(r,q),~l=e(s,q)$, $1\le k < l$, and either $l = 6$ and $k = 5$, or $l\in \{7, 14\}$ and $k\ge 3$, or $l = 9$ and
$k\ge 4$, or $l\in \{8, 12\}$ and $k\ge 5$, $k\ne 6$, or $l = 10$ and $k\ge 3$, $k\not\in\{4,6\}$, or $l = 18$ and
$k\not\in\{1, 2, 6\}$, or $l = 20$ and $r\cdot k\ne20$, or $l\in \{15, 24, 30\}$.
    \end{enumerate}

In particular, the compact form for $\Gamma(E_8(q))$ is the following. Here, $$R (q)= R_1(q) \cup R_2(q) \cup \{p\}$$ and the vector from $5$ to $R_4(q)$ and the dotted edge $\{5, R_{20}(q)\}$ indicate that $R_4(q)$ and $R_{20}(q)$ are not adjacent, but if $5\in R_4(q)$ {\rm(}i.e., $q^2\equiv -1 \pmod {5}${\rm)}, then there exist edges between $5$ and the primes from $R_{20}(q)$.
\usetikzlibrary {arrows.meta}
 \centering{
	\begin{tikzpicture}
		\tikzstyle{every node}=[draw,circle,fill=black,minimum size=4pt,
		inner sep=0pt]
		
		\draw (0,0) node (1) [label=below:$\mbox{ }\mbox{ }R(q)$]{}
		++ (-15:3.0cm) node (18) [label=below:$R_{18}(q)$]{}
        ++ (80:1.4cm) node (5) [label=right:$\mbox{ } \mbox{ } R_{5}(q)$]{}
        ++ (120:1.4cm) node (3) [label=right:$R_{3}(q)$]{}
        ++ (160:1.4cm) node (8) [label=above:$R_{8}(q)$]{}
        ++ (180:1.9cm) node (12) [label=above:$R_{12}(q)$]{}
        ++ (200:1.4cm) node (6) [label=left:$R_{6}(q)$]{}
        ++ (-125:1.4cm) node (10) [label=left:$R_{10}(q)$]{}
        ++ (-90:1.9cm) node (9) [label=left:$R_{9}(q)$]{}
        ++ (-35:2.2cm) node (14) [label=left:$R_{14}(q)$]{}
        ++ (0:2.9cm) node (7) [label=right:$R_{7}(q)$]{}
        ++ (40:3.0cm) node (4) [label=below:$R_{4}(q)$]{}
        ++ (10:3.0cm) node (55) [label=above:$5$]{}
        ++ (0:1.0cm) node (20) [label=right:$R_{20}(q)$]{}
		++ (-100:1.5cm) node (15) [label=right:$R_{15}(q)$]{}
		++ (100:3.0cm) node (24) [label=right:$R_{24}(q)$]{}
		++ (100:1.5cm) node (30) [label=right:$R_{30}(q)$]{}
        (1)--(7)
		(1)--(14)
		(4)--(8)
		(10)--(6)
		(10)--(4)
		(10)--(1)
		(1)--(6)
		(6)--(4)
		(6)--(18)
		(6)--(8)
		(6)--(12)
		(6)--(3)
		(1)--(18)
		(1)--(8)
		(3)--(8)
		(1)--(12)
		(4)--(12)
		(3)--(12)
		(1)--(9)
		(3)--(9)
		(1)--(3)
		(4)--(3)
		(5)--(3)
		(5)--(1)
            (1)--(4)
		(5)--(4);
        \draw[arrows = {-Latex[width=6pt, length=6pt]}]
		(55)--(4);
		\draw[dotted]
		(55)--(20)
		;
	\end{tikzpicture}}
\end{lemm}

\begin{lemm}[{\rm \cite[Lemma~2.16]{MasPansStar}}]\label{E8nontrivK_R24}
Let $G$ be a group with a non-trivial nilpotent normal subgroup $K$ such that $G/K$ has a subgroup $H$ isomorphic to $E_8(q)$, where $q$ is a prime power. Then $R_{24}(q) \subset \pi_1(G)$.
\end{lemm}

\begin{lemm}\label{OrdersAndGraphs}
$(1)$ If $S\cong E_8(2)$, then $|\pi(S)|=16$.

$(2)$ If $S \cong E_8(3)$, then $|\pi(S)|=19$,
$$\pi_1(S)=\{2, 3, 5, 7, 11, 13, 19, 37, 41, 61, 73, 547, 757, 1093, 1181\},$$ $$\mbox{ } \pi_2(S)=\{4561\}, \mbox{ } \pi_3(S)=\{6481\}, \mbox{ and } \pi_4(S)=\{31, 271\};$$
\begin{itemize}
\item[] ${\rm deg}(2)={\rm deg}(3)=13$,
\item[] ${\rm deg}(5)={\rm deg}(7)=9$,
\item[] ${\rm deg}(13)=8$,
\item[] ${\rm deg}(41)={\rm deg}(73)=5$,
\item[] ${\rm deg}(11)={\rm deg}(19)={\rm deg}(37)={\rm deg}(61)=4$,
\item[] ${\rm deg}(757)=3$,
\item[] ${\rm deg}(547)={\rm deg}(1093)=2$,
\item[] ${\rm deg}(31)={\rm deg}(271)={\rm deg}(1181)=1$,
\item[] ${\rm deg}(4561)={\rm deg}(6481)=0$.
\end{itemize}

\smallskip
$(3)$ If $S \cong E_8(4)$, then $|\pi(S)|=26$,
$$\pi_1(S)=\{2, 3, 5, 7, 11, 13, 17, 19, 29, 31, 37, 41, 43,  73, 109, 113, 127,  241, 257\},$$
$$\mbox{ } \pi_2(S)=\{151, 331\}, \mbox{ } \pi_3(S)=\{61,1321\}, \mbox{ } \pi_4(S)=\{97,673\}, \mbox{ and }\pi_5(S)=\{61681\}.$$

\smallskip

$(4)$  If $S \cong E_8(5)$, then $|\pi(S)|=24$,
$$\pi_1(S)=\{2, 3, 5, 7, 11, 13, 19,  29, 31, 71, 313, 449, 521, 601, 829, 5167, 19531\},$$
$$\mbox{ } \pi_2(S)=\{181, 1741\}, \mbox{ } \pi_3(S)=\{61,7621\}, \mbox{ } \pi_4(S)=\{390001\}, \mbox{ and }\pi_5(S)=\{41,9161\}.$$

\smallskip

$(5)$ If $S \cong E_8(7)$, then $|\pi(S)|=27$,
$$\pi_1(S)=\{2, 3, 5, 7, 11, 13, 19, 29,  37, 43,  113, 181, 191,$$  $$281, 911, 1063, 1201, 2801, 4021, 4733, 117307\},$$
$$\mbox{ } \pi_2(S)=\{31,159871\}, \mbox{ } \pi_3(S)=\{73,193,409\}, \mbox{ and } \pi_4(S)=\{6568801\};$$
\begin{itemize}
\item[] ${\rm deg}(2)={\rm deg}(3)={\rm deg}(7)=18$,
\item[] ${\rm deg}(5)=13$,
\item[] ${\rm deg}(19)={\rm deg}(43)=11$,
\item[] ${\rm deg}(13)={\rm deg}(181)=7$,
\item[] ${\rm deg}(11)={\rm deg}(191)={\rm deg}(1201)=6$,
\item[] ${\rm deg}(37)={\rm deg}(1063)={\rm deg}(2801)=5$,
\item[] ${\rm deg}(29)={\rm deg}(113)={\rm deg}(911)={\rm deg}(4733)={\rm deg}(117307)=4$,
\item[] ${\rm deg}(73)={\rm deg}(193)={\rm deg}(281)={\rm deg}(409)={\rm deg}(4021)=2$,
\item[] ${\rm deg}(31)={\rm deg}(159871)=1$,
\item[] ${\rm deg}(6568801)=0$.
\end{itemize}

\smallskip

$(6)$ If $S \cong E_8(8)$, then $|\pi(S)|=29$,
$$\pi_1(S)=\{2, 3, 5, 7, 11, 13, 17, 19, 31, 37, 41, 43, 61, 73, 109, 127,$$ $$ 151, 241, 331, 337, 1321, 5419, 87211, 262657 \},$$
$$\mbox{ } \pi_2(S)=\{631,23311\}, \mbox{ } \pi_3(S)=\{433, 38737\}, \mbox{ and } \pi_4(S)=\{18837001\};$$
\begin{itemize}
\item[] ${\rm deg}(2)={\rm deg}(3)={\rm deg}(7)=20$,
\item[] ${\rm deg}(5)=17$,
\item[] ${\rm deg}(13)=14$,
\item[] ${\rm deg}(19)={\rm deg}(73)=13$,
\item[] ${\rm deg}(17)={\rm deg}(37)={\rm deg}(109)={\rm deg}(241)=8$,
\item[] ${\rm deg}(11)={\rm deg}(31)={\rm deg}(151)={\rm deg}(331)=7$,
\item[] ${\rm deg}(43)={\rm deg}(127)={\rm deg}(337)={\rm deg}(5419)={\rm deg}(87211)={\rm deg}(262657)=4$,
\item[] ${\rm deg}(41)={\rm deg}(61)={\rm deg}(1321)=3$,
\item[] ${\rm deg}(433)={\rm deg}(631)={\rm deg}(23311)={\rm deg}(38737)=1$,
\item[] ${\rm deg}(18837001)=0$.
\end{itemize}

\smallskip

$(7)$ If $S \cong E_8(9)$, then $|\pi(S)|=29$,
$$\pi_1(S)=\{2, 3, 5, 7, 11, 13, 17, 19, 29, 37, 41, 61, 73, 193, 547, 757, 1093, 1181, 6481, 16493, 530713\},$$
$$\mbox{ } \pi_2(S)=\{31, 271, 4561\}, \mbox{ } \pi_3(S)=\{47763361\}, \mbox{ }$$ $$ \pi_4(S)=\{97, 577, 769\}, \mbox{ and }\pi_5(S)=\{42521761\}.$$

\smallskip

$(8)$ If $S \cong E_8(17)$, then $|\pi(S)|=28$,
$$\pi_1(S)=\{2, 3, 5, 7, 11, 13, 17, 19, 29, 71, 101, 307, 1423, 5653, $$ $$21881, 41761, 63541, 83233, 88741, 1270657, 22796593, 25646167\},$$
$$\mbox{ } \pi_2(S)=\{6566760001\}, \mbox{ } \pi_3(S)=\{31, 238212511\}, \mbox{ and } \pi_4(S)=\{73, 1321, 72337\};$$
\begin{itemize}
\item[] ${\rm deg}(2)={\rm deg}(3)={\rm deg}(17)=19$,
\item[] ${\rm deg}(5)=15$,
\item[] ${\rm deg}(7)={\rm deg}(13)=14$,
\item[] ${\rm deg}(29)=13$,
\item[] ${\rm deg}(307)=12$,
\item[] ${\rm deg}(11)={\rm deg}(71)={\rm deg}(101)=9$,
\item[] ${\rm deg}(41761)={\rm deg}(83233)=8$,
\item[] ${\rm deg}(1423)={\rm deg}(5653)={\rm deg}(88741)=6$,
\item[] ${\rm deg}(19)={\rm deg}(1270657)=5$,
\item[] ${\rm deg}(22796593)={\rm deg}(25646167)=3$,
\item[] ${\rm deg}(73)={\rm deg}(1321)={\rm deg}(72337)={\rm deg}(21881)={\rm deg}(63541)=2$,
\item[] ${\rm deg}(31)={\rm deg}(238212511)=1$
\item[] ${\rm deg}(6566760001)=0$.
\end{itemize}

\end{lemm}

\medskip

\begin{lemm}\label{Primes20} Let $L=E_8(q)$. The following statements hold{\rm:}

$(1)$ if $q >3$ is a prime power, then $$|\pi_1(L)|\ge \begin{cases} |\pi(q\cdot (q^2-1))|+11 \mbox{ if } q\equiv 0, \pm 1 \pmod{5};\\
|\pi(q\cdot (q^2-1))|+12 \mbox{ if } q\equiv \pm 2 \pmod{5};
 \end{cases}$$
and in $\Gamma(L)$ each vertex from $\pi(q\cdot (q^2-1))$ has degree at least $|\pi(q\cdot (q^2-1))|+10$.

$(2)$ if $q=q_0^r$, where $r \ge 11$ is a prime, then $$|\pi_1(L)|\ge |\pi(q\cdot (q^2-1))| +17$$ and in $\Gamma(L)$ each vertex from $\pi(q\cdot (q^2-1))$ has degree at least $|\pi(q\cdot (q^2-1))|+16$.

\end{lemm}

\begin{proof} Statement~$(1)$ follows directly from Lemmas~\ref{graphE8} and~\ref{zsigm}.

If $r\ge 11$ is a prime, then the set $$\mathfrak{I}=\{3, 4, 5, 7, 8, 9, 10, 12, 14, 18, 3r, 4r, 6r, 8r, 9r, 12r, 18r\}$$ consists of pairwise distinct numbers. Moreover, $|\pi(q\cdot (q^2-1))| \cap R_i(q_0)=\varnothing$ for $i \in \mathfrak{I}$. Now statement~$(3)$ follows directly from Lemmas~\ref{graphE8} and~\ref{zsigm}.

\end{proof}

\section{Proof of Main theorem for the group ${^2}E_6(2)$}

Let $L={^2}E_6(2)$. By~\cite{Atlas}, $\Gamma(L)$ is as in Figure~2.

\begin{center}    \begin{tikzpicture}
        \tikzstyle{every node}=[draw,circle,fill=white,minimum size=4pt, inner sep=0pt]
        \draw (0,0) node (3) [label=below:$3$] {}
        (-3.2cm:-1.7cm) node (2) [label=above:$2$] {}
        (0.0cm:-2.0cm) node (5) [label=left:$5$] {}
        (-1.4cm:-2.65cm) node (7) [label=left:$7$] {}
        (1.0cm:1.5cm) node (11) [label=below:$11$] {}

        (6.0cm:4.0cm) node (17) [label=left:$17$] {}
        (0.0cm:-4.0cm) node (19) [label=left:$19$] {}
        (-7.0cm:4.2cm) node (13) [label=left:$13$] {}
 (3)--(11)
 (2)--(11)
 (2)--(3)
 (2)--(5)
 (2)--(7)
 (3)--(5)
 (3)--(7)
 (5)--(7)
 ;
    \end{tikzpicture}\\
    Figure~2
    \end{center}

Let $G$ be a group such that $\Gamma(G) \cong \Gamma(L)$. Then $|\pi(G)|=8$, $s(G)=4$ and $t(G)=5$.
By Lemmas~\ref{Gruenberg--Kegel theorem} and~\ref{Structure}, $G$ has a normal nilpotent subgroup $N$  such that $G/N$ is an almost simple group with simple non-abelian socle $S$ such that $s(S)\ge 4$.

Suppose for the contradiction that $S \not \cong {^2}E_6(q)$. Note that $|\pi(S)|\le 8$. By Lemma~\ref{Table}, $S$ is contained in the following list{\rm:} $PSL_3(4)$, $^2B_2(q)$ for $q=2^{2m+1}>2$,  $Ly$, $M_{22}$, $J_1$, $O'N$, $E_8(q)$ for some $q$.

Let $S\cong PSL_3(4)$. Then by~\cite{Atlas}, $\Gamma(S)$ consists of the following singleton connected components $\{2\}$, $\{3\}$, $\{5\}$, and $\{7\}$. Since $|Out(S)|=12$ and $|\pi(G)|=8$, there are four pairwise distinct primes $p$, $q$, $r$, and $t$ such that $\{p, q, r, t\}\subseteq \pi(N)\setminus \pi(G/N)$. Since $N$ is nilpotent, the primes $p$, $q$, $r$, and $t$ are pairwise adjacent in $\Gamma(N)$, therefore, they are pairwise adjacent in $\Gamma(G)$. By Lemma~\ref{vas}, vertex $2$ is adjacent to each of these primes. Thus, $\Gamma(G)$ has a $5$-clique; a contradiction.

Let $S \cong {^2B_2(q)}$ for $q=2^{2m+1}>2$. By Lemma~\ref{Table}, $\Gamma(S)$ has four connected components{\rm:} $\{2\}$, $\pi(q-1)$, $\pi(q-\sqrt{2q}+1)$, and $\pi(q+\sqrt{2q}+1)$. By Lemma~\ref{Structure}, for $i\in \{2, 3, 4\}$, $|\pi_i(S)|=1$, therefore, $|\pi(S)|=4$. Thus, by \cite[Theorem~1]{Bugeaud}, $q \in \{2^3, 2^5\}$. By~\cite{Atlas}, $|Out({^2}B_2(2^3))|=3$ and $|Out({^2}B_2(2^5))|=5$. If $G/N \cong Aut({^2}B_2(2^3))$, then by~\cite{Atlas}, $\Gamma(G/N)$ has three connected components $\{2, 3, 5\}$, $\{7\}$, and $\{13\}$ and because $\pi(N)\subseteq \pi_1(G)$ by Lemma~\ref{Gruenberg--Kegel theorem}, $\Gamma(G)$ has at most three connected components; a contradiction. Thus, $S\cong {^2}B_2(2^5)$ or $G/N\cong {^2}B_2(2^3)$. Again $|\pi(G/N)|=4$ and $|\pi(G)|=8$, therefore, there are four pairwise distinct primes $p$, $q$, $r$, and $t$ such that $\{p, q, r, t\}\subseteq \pi(N) \setminus \pi(G/N)$. Since $N$ is nilpotent, the primes $p$, $q$, $r$, and $t$ are pairwise adjacent in $\Gamma(N)$, therefore, they are pairwise adjacent in $\Gamma(G)$. By Lemma~\ref{vas}, vertex $2$ is adjacent to each of these primes. Thus, $\Gamma(G)$ has a $5$-clique; a contradiction.

Let $S \cong Ly$. By~\cite{Atlas}, $|\pi(S)|=8$ and $\Gamma(S)$ is as in Figure~3.

\begin{center}    \begin{tikzpicture}
        \tikzstyle{every node}=[draw,circle,fill=white,minimum size=4pt, inner sep=0pt]
        \draw (0,0) node (3) [label=below:$3$] {}
        (-3.2cm:-1.7cm) node (2) [label=above:$2$] {}
        (0.0cm:-2.0cm) node (5) [label=left:$5$] {}
        (-1.4cm:-2.65cm) node (7) [label=left:$7$] {}
        (1.0cm:1.5cm) node (11) [label=below:$11$] {}

        (6.0cm:4.0cm) node (31) [label=left:$31$] {}
        (0.0cm:-4.0cm) node (37) [label=left:$37$] {}
        (-7.0cm:4.2cm) node (67) [label=left:$67$] {}
 (3)--(11)
 (2)--(11)
 (2)--(3)
 (2)--(5)
 (2)--(7)
 (3)--(5)
 (3)--(7)
 ;
    \end{tikzpicture}\\
    Figure~3
    \end{center}
$\Gamma(L)$ has eight edges while $\Gamma(S)$ has seven edges. The vertices $2$ and $3$ are adjacent to all the vertices in $\pi_1(S)$, therefore, in $\Gamma(G)$ there is an edge between $5$ and $7$, $5$ and $11$ or $7$ and $11$. By~\cite{Atlas}, $|Out(S)|=1$. Thus, by Lemma~\ref{NormalSeriesAdj}, $O_p(G)\not =1$ for some $p \in \{5, 7, 11\}$. By~\cite{Atlas}, $S$ has a maximal subgroup $M\cong G_2(5)$ and $\pi(M)=\{2, 3, 5, 7, 31\}$. If $p=5$, then by Lemma~\ref{Unisingular}, the vertices $5$ and $31$ are adjacent in $\Gamma(G)$, therefore, $\Gamma(G)$ has at most three connected components; a contradiction. If $p\in \{7, 11\}$, then by Lemmas~\ref{Eigenvector1} and~\ref{TiepZalThm}, the vertices $p$ and $31$ are adjacent in $\Gamma(G)$, therefore, again $\Gamma(G)$ has at most three connected components; a contradiction.

Let $S \cong M_{22}$. By~\cite{Atlas}, $|\pi(S)|=5$, $\Gamma(S)$ has four connected components $\{2, 3\}$, $\{5\}$, $\{7\}$, and $\{11\}$, and $|Out(S)|=2$. Thus, there are three pairwise distinct primes $p$, $q$, and $r$ such that $\{p, q, r\} \subseteq \pi(N) \setminus \pi(G/N)$. By Lemma~\ref{vas}, vertex $2$ is adjacent in $\Gamma(G)$ to each odd prime from $\pi(N)$. If $3$ divides $|N|$, then $\{2, 3, p, q, r\}$ is a $5$-clique in $\Gamma(G)$; a contradiction. Thus, we can assume that $3$ does not divide $|N|$. By~\cite{Atlas}, $S$ has a maximal subgroup $M \cong PSL_2(11)$ and $3$ divides $|M|$. Then by Lemma~\ref{PSL2qODD}, vertex $3$ is adjacent in $\Gamma(G)$ to each prime from the set $\{p, q, r\}$. Thus, again $\{2, 3, p, q, r\}$ is a $5$-clique in $\Gamma(G)$; a contradiction.

Let $S \cong J_1$. By~\cite{Atlas}, $|Out(S)|=1$ and $\Gamma(S)$ consists of four cliques $\{2, 3, 5\}$, $\{7\}$, $\{11\}$, $\{19\}$. Thus, there are two distinct primes $p$ and $q$ such that $\{p, q\} \subseteq \pi(N) \setminus \pi(G/N)$. By Lemma~\ref{vas}, vertex $2$ is adjacent to each odd prime from $\pi(N)$. By~\cite{Atlas}, $S$ has a maximal subgroup $M \cong PSL_2(11)$. By Lemma~\ref{PSL2qODD}, vertex $3$ is also adjacent both to $p$ and $q$. If $11 \in \pi(N)$, then $\Gamma(G)$ has at most three connected components; a contradiction. Thus, $11$ does not divide $|N|$. By~\cite{Atlas}, $S$ has a maximal subgroup $T \cong 11:5$ which is a Frobenius group with a cyclic complement of order $5$. By Lemma~\ref{MazurovLemma}, $11 \in \pi_1(G)$ or $\{2, 3, 5, p, q\}$ is a $5$-clique in $\Gamma(G)$; a contradiction.

Let $S\cong O'N$. By~\cite{Atlas}, $|Out(S)|=2$. Because $|O'N|=7$, there exists a prime $p \in \pi(N)\setminus \pi(G/N)$. From Lemma~\ref{BrChar} and the ordinary character table for $S$ (see, for example,~\cite{Atlas}) it follows that in every faithful irreducible module of $S$ in characteristic $p$ every element of order $13$, $17$ or $19$ fixes a non-zero vector (see also~\cite[4.3]{LeePopiel_2}). Thus, $\Gamma(G)$ has at most three connected components; a contradiction.

Let $S \cong E_8(q)$ for some $q$. Then by Lemmas~\ref{OrdersAndGraphs} and~\ref{Primes20}, $|S|>11$; a contradiction.

\smallskip

Thus, we proved that $S\cong L$. This implies that $\Gamma(S)=\Gamma(L)=\Gamma(G)$, therefore, $G \cong L$ by \cite[Theorem]{Kondrat'ev2E6(2)}. \hfill $\Box$

\section{Proof of Main theorem for the groups $E_8(q)$, where $q\in \{4,5,9\}$}

Let $L=E_8(q)$, where $q \in \{4,5,9\}$, and let $G$ be a group such that $\Gamma(G)\cong \Gamma(L)$. By~\cite[Theorem~1]{Zavarnitsine_2013}, $G\cong E_8(q_1)$, where $q_1 \equiv 0, \pm 1 \pmod{4}$. By Lemma~\ref{graphE8}, a vertex $x$ is adjacent to each vertex in $\pi_1(G)$ if and only if $x\in \pi(q_1\cdot (q_1^2-1))$; a vertex $y$ is adjacent to each vertex in $\pi_1(L)$ if and only if $y\in \pi(q\cdot (q^2-1))$. Thus, $|\pi(q_1^2-1)|=|\pi(q^2-1)|=2$. By Lemma~\ref{PrimesSmall}, $q_1\in \{4, 5, 9\}$. Because by Lemma~\ref{OrdersAndGraphs}, the numbers $|\pi(E_8(4))|=26$, $|\pi(E_8(5))|=24$, and $|\pi(E_8(9))|=29$ are pairwise distinct, we have $G\cong L$.

\section{Proof of Main theorem for the group $E_8(q)$, where $q\in \{3,7,8, 17\}$}

Let $L=E_8(q)$, where $q\in \{3, 7, 8, 17\}$, let $G$ be a group such that $\Gamma(G) \cong \Gamma(L)$, and let $$\Psi: \Gamma(L) \rightarrow \Gamma(G) \mbox{ be a graph isomorphism}.$$ By Lemmas~\ref{OrdersAndGraphs}, \ref{Gruenberg--Kegel theorem} and~\ref{Structure}, $G$ has a normal nilpotent subgroup $N$  such that $G/N$ is an almost simple group with simple non-abelian socle $S$ such that $s(S)\ge 4$. Define by $\overline{\phantom{x}}: G \rightarrow N$ the natural epimorphism from $G$ to $G/N$. By Lemma~\ref{Table}, choices for $S$ are listed in Table~$1$. Lemmas~\ref{Structure} and~\ref{OrdersAndGraphs} imply that one of the following statements hold{\rm:}

$(i)$ $S\cong {^2}B_2(q_1)$, where $q_1=2^{2m+1}\ge 8${\rm;}

$(ii)$ $S \cong E_8(q_1)$ for some prime power $q_1$.

\medskip

Assume that statement $(i)$ holds. By Table~$1$, $\Gamma(S)$ has four connected components{\rm:} $$\{2\}, \mbox{   }\mbox{   } \pi(q_1-1), \mbox{   }\mbox{   } \pi(q_1-\sqrt{2q_1}+1), \mbox{   }\mbox{   } \mbox{ and } \mbox{   }\mbox{   }\pi(q_1+\sqrt{2q_1}+1).$$
By Lemma~\ref{Structure}, $\Psi(\pi_1(L)) = \pi_1(G)$. Note that $\pi_1(S)=\{2\}$, therefore, $$\pi_1(G)=\pi(N) \cup \{2\} \cup \pi(\overline{G}/S).$$ By Lemma~\ref{vas}, the vertex $2$ is adjacent to each odd prime from $\pi(N) \cup \pi(\overline{G}/S)$. Thus, in $\Gamma(G)$ the vertex $2$ has degree $|\pi_1(G)|-1$; a contradiction to Lemma~\ref{OrdersAndGraphs}.

\medskip

Assume that statement $(ii)$ holds. Let $R=\pi(q_1\cdot (q_1^2-1))$.

\smallskip

Let $q=3$. By Lemmas~\ref{Structure} and~\ref{Primes20}, we have $|R|\le 3$. Lemmas~\ref{PrimesSmall} and~\ref{OrdersAndGraphs} imply that $q_1 \in \{2, 3\}$. If $q_1=2$, then $|\pi(N)|=3$, therefore, by Lemma~\ref{E8nontrivK_R24}, $s(G) \le 3$; a contradiction. Thus, $q_1=3$ and $\Gamma(G)=\Gamma(L)$. By~\cite[Theorem~6.1]{MasPansStar}, $G \cong E_8(q_0)$ for some prime power $q_0$, and we immediately conclude that $G\cong L$.

\smallskip

Let $q=7$. If $|R|\ge 4$, then by Lemma~\ref{Primes20}, in $\Gamma(S)$ each vertex from $R$ has degree at least $|R|+10\ge 14$. Thus, in $\pi_1(S)$ there are at least $4$ vertices of degree at least $14$; a contradiction to Lemma~\ref{OrdersAndGraphs}. Thus, $|R| \le 3$ and $|S| \le |E_8(7)|$, therefore, $q_1 \in \{2, 3, 4, 5, 7\}$ by Lemma~\ref{PrimesSmall}. But by Lemma~\ref{OrdersAndGraphs}, $\Gamma(E_8(q))$ for $q \in \{2, 3, 4, 5\}$ do not contain a connected component of order $3$, therefore,  by Lemma~\ref{Structure}, $S\cong E_8(7)$ and $\Gamma(G)=\Gamma(E_8(7))$. By~\cite[Theorem~6.1]{MasPansStar}, $G\cong E_8(u)$ with $u \equiv \pm 2\pmod{8}$ and $|\pi(u^2-1)| \le 2$. Thus, $u \in \{2, 3, 7, 17\}$. Now Lemma~\ref{OrdersAndGraphs} implies that $G \cong E_8(7)$.

\smallskip

Let $q \in \{8, 17\}$. Suppose that $N \not = 1$. By Lemma~\ref{E8nontrivK_R24} we have $4=s(G) \le s(S)-1$, therefore, we conclude that $s(S)=5$ and from Lemma~\ref{Table} we conclude that $q_1\equiv 0, \pm 1 \pmod{5}$. Moreover, by Lemma~\ref{E8nontrivK_R24} we have $$\pi_1(G)=\pi_1(S) \cup \pi(N) \cup \pi(\overline{G}/S) \cup R_{24}(q_1).$$ By Lemma~\ref{graphE8}, each vertex from $R$ is adjacent to each other vertex from $\pi_1(S)$.

Suppose that $|R|\ge 4$. By Lemma~\ref{Primes20}, $\overline{G}/S$ is a $\{2, 3, 5, 7\}$-group or $|\pi_1(S)| \ge 21$. In the latter case, $\Gamma(G)$ contains at least $4$ vertices of degree at least $20$; a contradiction to Lemma~\ref{OrdersAndGraphs}. Thus, since $\{2, 3, 5, 7\} \subseteq \pi_1(S)$ by Fermat's little theorem we have  $$\pi_1(G)=\pi_1(S) \cup \pi(N) \cup R_{24}(q_1).$$
Moreover, by~\cite[Table~5.1]{LiSaSe92}, $S$ has a subgroup isomorphic to ${^3}D_4(q)\times {^3}D_4(q)$. By~\cite[Theorem~10.3.1]{Gorenstein}, each Sylow $r$-subgroup of $S$ for $r \in R$ can not act fixed-point-freely on any group of order coprime to $r$, therefore, each vertex from $R$ is adjacent to each vertex from $\pi_1(S) \cup \pi(N)$. Thus, $\Gamma(G)$ has at least $4$ vertices of degree at least $|\pi_1(S) \cup \pi(N)|-1$. Therefore, if $q=8$, then $|\pi_1(S) \cup \pi(N)|\le 18$, and if $q=17$, then $|\pi_1(S) \cup \pi(N)|\le 16$. By Lemma~\ref{Structure}, in both cases we have $|\pi_1(G) \setminus (\pi_1(S) \cup \pi(N))| \ge 6$ and $\pi_1(G) \setminus (\pi_1(S) \cup \pi(N)) \subseteq R_{24}(q_1)$. Remind that $R_{24}(q_1)$ is a clique in $\Gamma(S)$ by Lemma~\ref{Table}. Thus, there are at least $6$ vertices in $\pi_1(G) \setminus (\pi_1(S) \cup \pi(N))$, and each of these vertices has degree at least $5$ in $\Gamma(S)$. Moreover, if $|R|\ge 4$, then by Lemma~\ref{OrdersAndGraphs}, there are at least $13$ vertices in $\pi_1(S)$ such that each of these vertices has degree at least $5$ in $\Gamma(S)$. Thus, $\Gamma(G)$ has at least $19$ vertices of degree at least $5$; a contradiction to Lemma~\ref{OrdersAndGraphs}.

We conclude that $|R| \le 3$, therefore, $q_1 \in \{4, 5, 9\}$ by Lemma~\ref{PrimesSmall}. By Lemma~\ref{OrdersAndGraphs}, $\Gamma(E_8(9))$ does not contain a connected component of order $2$ while $\Gamma(L)$ has such a connected component; a contradiction to Lemma~\ref{Structure}. Thus, $q_1\in \{4, 5\}$ and $\pi(G)=\pi(N) \cup \pi(S)$. By Lemma~\ref{OrdersAndGraphs}, if $q=5$, then $|\pi(S)|=24$ and $|\pi_1(S)|=17$; if $q=4$, then $|\pi(S)|=26$ and $|\pi_1(S)|=19$. By Lemma~\ref{vas}, vertex $2$ is adjacent to each odd vertex from $\pi(N)$, therefore, in $\Gamma(G)$  vertex $2$ has degree at least $21$ if $q_1=8$ and has degree at least $20$ if $q_1=17$; a contradiction to Lemma~\ref{OrdersAndGraphs}.

\smallskip

Thus, we have proved that $N=1$. Again suppose that $|R|\ge 4$. By Lemma~\ref{Primes20}, $\overline{G}/S$ is a $\{2, 3, 5, 7\}$-group or $\Gamma(G)$ contains at least $4$ vertices of degree at least $20$; in the latter case we obtain immediately a contradiction with Lemma~\ref{OrdersAndGraphs}. Note that $\{2, 3, 5, 7\} \subseteq \pi_1(S)$ by Fermat's little theorem. Thus, one of the following statements holds{\rm:}

\smallskip

$(a)$ $s(S)=4$, $\pi_1(S)=\pi_1(G)$, and if $r \in \pi_1(G)$ and $2$ and $r$ are non-adjacent in $\Gamma(G)$, then $r \in R_{20}(q_1)${\rm;}

\smallskip

$(b)$ $s(S)=5$ and there exists $i \in \{15, 20, 24, 30\}$ such that $\pi_1(G)=\pi_1(S) \cup R_i(q_1)$, moreover, if $r \in \pi_1(G)$ and $2$ and $r$ are non-adjacent in $\Gamma(G)$, then $r \in R_i(q_1)$.

\smallskip

Thus, we see that all the vertices in $\pi_1(G)$ which are non-adjacent to $2$ in $\Gamma(G)$ form a clique $Q$. Moreover, each vertex from $R$ is adjacent to each vertex from $\pi_1(G)\setminus Q$, therefore, $\Gamma(G)$ has at least $4$ vertices of degree at least $|\pi_1(G)\setminus Q|-1$.
Thus, if $q=8$, then $|\pi_1(G)\setminus Q|\le 18$, and if $q=17$, then $|\pi_1(G)\setminus Q|\le 16$. By Lemma~\ref{Structure}, in both cases we have $|Q| \ge 6$, therefore, there are at least $6$ vertices in $Q$, and each of these vertices has degree at least $5$ in $\Gamma(S)$. Moreover, if $|R|\ge 4$, then by Lemma~\ref{OrdersAndGraphs}, there are at least $13$ vertices in $\pi_1(S)$ such that each of these vertices has degree at least $5$ in $\Gamma(S)$. Thus, $\Gamma(G)$ has at least $19$ vertices of degree at least $5$; a contradiction to Lemma~\ref{OrdersAndGraphs}.

We again conclude that $|R| \le 3$, therefore, $q_1 \in \{2, 3, 4, 5, 7, 8, 9, 17\}$ by Lemma~\ref{PrimesSmall} and $|\pi(G)|=|\pi(S)|$. Now Lemma~\ref{OrdersAndGraphs} implies that if $q_1\not = q_2$ and $q_1, q_2 \in \{2, 3, 4, 5, 7, 8, 9, 17\}$, then $\{q_1,q_2\}=\{8,9\}$. But the graph $\Gamma(E_8(8))$ has connected components of orders $24$, $2$, $2$, and $2$ while $\Gamma(E_8(9))$ has connected components of orders $21$, $3$, $3$, $1$, $1$. Using Lemma~\ref{Structure} and \cite[Theorem~6.1]{MasPansStar} we conclude that $G \cong L$.

\section{Acknowledgements}

We dedicate this paper to Prof. Cheryl E. Praeger, a great female specialist in Group Theory and Combinatorics who gives us an inspiring example that women can be remarkable specialists in Mathematics and at the same time can have happy families and children.

\medskip

The first author is supported by the National Natural Science Foundation of China (project No.~12461063). The work of the second author was performed as part of research conducted in the Ural Mathematical Center with the financial support of the Ministry of Science and Higher Education of the Russian Federation (Agreement number 075-02-2025-1549). This research continues the project which was supported by the Russian Science Foundation (project no. 19-71-10067).


\begin{thebibliography}{99}


\bibitem{Asch86}
M.~Aschbacher, \textit{Finite Group Theory}, Cambridge Univ. Press, Cambridge, 1986.


\bibitem{ak}
O.~A.~Alekseeva and A.~S.~Kondrat'ev, On recognizability of some finite simple orthogonal groups by spectrum, \textit{Proc. Steklov Inst. Math.}, \textbf{266} (2009), 10--23.


\bibitem{Bang}
A.~S.~Bang, Taltheoretiske Unders\o{}gelser, \textit{Tidsskrift Math.}, \textbf{4} (1886), 70-–80.


\bibitem{Bugeaud}
Y.~Bugeaud, Z.~Cao, M.~Mignotte, On simple $K_4$-groups, \textit{J. Algebra}, \textbf{241}:2 (2001), 658--668.

\bibitem{CaMas}
P.~J.~Cameron, N.~V.~Maslova, Criterion of unrecognizability of a finite group by its Gruenberg-Kegel graph, \textit{J. Algebra}, {\bf 607}: Part~A (2022), 186--213.



\bibitem{Atlas} J.~H.~Conway, et. al., {\it Atlas of finite groups},  Clarendon Press, Oxford, 1985.


\bibitem{DoJaLu}
S.~Dolfi, E.~Jabara, M.~S.~Lucido, $C55$-groups, \textit{Siberian Math. J.}, \textbf{45}:6 (2004), 1053--1062.


\bibitem{Gorenstein}
D.~Gorenstein, \textit{Finite groups}, Harper and Row, N. Y., 1968.

\bibitem{GreMazShiVasYang}
M.~A.~Grechkoseeva, V.~D.~Mazurov, W.~Shi, A.~V.~Vasil'ev, N.~Yang, Finite groups isospectral to simple groups, \textit{Communications in Mathematics and Statistics}, \textbf{11} (2023), 169--194.

\bibitem{GuTi03}
R.~M.~Guralnick, P.~H.~Tiep, Finite simple unisingular groups of Lie type,
\textit{J. Group Theory}, \textbf{6}:3 (2003), 271--310.

\bibitem{Herzog}
M.~Herzog, On finite simple groups of order divisible by three primes only, \textit{J. Algebra}, \textbf{10}:3
(1968), 383--388.

\bibitem{KL}
P.~Kleidman, M.~Liebeck, \textit{The subgroup structure of the finite classical groups}, Cambridge University Press, Cambridge,  1990.


\bibitem{Kondrat'ev2E6(2)}
A.~S.~Kondrat'ev, Recognizability by prime graph of the group $^2E_6(2)$, \textit{J. Math. Sci. (New York)}, \textbf{259}:4 (2021), 463--466.

\bibitem{LeePopiel_2}
M.~Lee, T.~Popiel, Recognizability of the sporadic groups by the isomorphism types of their prime graphs, arXiv:2310.10113 [math.GR].

\bibitem{LiSaSe92}
M.~W.~Liebeck, J.~Saxl, G.~M.~Seitz,
{\it Subgroups of maximal rank in finite exceptional groups of Lie type}, Proc. London Math. Soc.{\rm(3)}, {\bf65} (1992), 297--325.


\bibitem{MaslovaPrGrSub}
N.~V.~Maslova, On the coincidence of Gruenberg–Kegel graphs of a finite simple group and its proper subgroup, \textit{Proc. Steklov Inst. Math.}, \textbf{288}:1 (2015), 129--141.


\bibitem{Maslova_surv}
N.~V.~Maslova, On arithmetical properties and arithmetical characterizations of finite groups, \textit{Proceedings of $(WM)^2$~--- World Meeting for Women in Mathematics~---$2022$}, to appear, arXiv:2401.04633 [math.GR].

\bibitem{Maslova_UWGTC2020}
N.~V.~Maslova, 2020 Ural Workshop on Group Theory and Combinatorics, \textit{Trudy Instituta Matematiki i Mekhaniki UrO RAN}, \textbf{27}:1 (2021), 273--282.


\bibitem{MasPansStar}
N.~V.~Maslova, V.~V.~Panshin, A.~M.~Staroletov, On characterization by Gruenberg--Kegel graph of finite simple exceptional groups of Lie type, \textit{European Journal of Mathematics}, \textbf{9} (2023), Article number: 78.


\bibitem{Mazurov}
V.~D.~Mazurov, Characterizations of finite groups by sets of orders of their elements, \textit{Algebra Logic}, \textbf{36}:1 (1997), 23--32.

\bibitem{Stewart}
W.~B.~Stewart, Groups having strongly self-centralizing $3$-centralizers, \textit{Proc. London Math. Soc.}, \textbf{426}:4 (1973), 653--680.

\bibitem{Thompson}
J.~Thompson, Finite groups with fixed-point-free automorphisms of prime order, \textit{Proceedings of the National Academy of Sciences of the United States of America}, \textbf{45}:4 (1959), 578--581.


\bibitem{TiepZal}
P.~H.~Tiep, A.~E.~Zalesski, Hall-Higman type theorems for exceptional groups of Lie type, I, \textit{J.~Algebra}, {\bf607} (2022), 755-794.

\bibitem{Va05}
A.~V.~Vasil$'$ev, On connection between the structure of finite group and properties of its prime graph, \textit{Siberian Math. J.}, {\bf46}:3 (2005), 396--404.

\bibitem{VaVd05} A.~V.~Vasil$'$ev, E.~P.~Vdovin, An adjacency criterion for the prime graph of a finite simple group,
\textit{Algebra Logic}, {\bf 44}:6 (2005), 381--406.

\bibitem{VaVd11}
A.~V.~Vasil$'$ev, E.~P.~Vdovin, Cocliques of maximal size in the prime graph of a finite simple group,
\textit{Algebra Logic}, {\bf 50}:4 (2011), 291--322.

\bibitem{Williams}
J.~S.~Williams, Prime graph components of finite groups, \textit{J. Algebra}, \textbf{69} (1981), 487--513.

\bibitem{Zavarnitsine_2006}
A.~V.~Zavarnitsine, Recognition of finite groups by the prime graph, \textit{Algebra and Logic}, \textbf{45} (2006), 220--231.

\bibitem{Zavarnitsine_2013}
A.~V.~Zavarnitsine, Finite groups with a five-component prime graph, \textit{Siberian Mathematical Journal}, \textbf{54} (2013), 40--46.


\bibitem{ZinMaz}
M.~R.~Zinov’eva, V.~D.~Mazurov, On finite groups with disconnected prime graph, \textit{Proc. Steklov Inst. Math.}, \textbf{283}:Suppl~1 (2013), 139--145.

\bibitem{Zs92}
K.~Zsigmondy,
{\it Zur Theorie der Potenzreste}, Monatsh. Math. Phys., {\bf3} (1892), 265--284.




\end{thebibliography}
\end{document}